\documentclass[11pt,a4paper]{article}

\usepackage[utf8x]{inputenc}

\usepackage{amsmath,amssymb,amstext,amsthm}
\usepackage{amsfonts}
\usepackage{graphicx,epstopdf,epsfig,multirow,epic,bm}
\usepackage{color}
\usepackage{multicol}
\usepackage{algorithm}
\usepackage{algorithmic}

\usepackage{paralist}

\theoremstyle{definition}
\newtheorem{thm}{Theorem}

\newtheorem{prop}[thm]{Proposition}
\theoremstyle{definition}
\newtheorem{defn}{Definition}
\theoremstyle{definition}
\newtheorem{rem}{Remark}

\theoremstyle{definition}

\theoremstyle{definition}

\theoremstyle{definition}

\theoremstyle{definition}

\DeclareGraphicsExtensions{.eps,.eps.gz}
\normalsize \rm
\newcommand{\qqed}{\hfill $\square$}

\begin{document}

\thispagestyle{empty}

\begin{center}
{\huge  \textbf{Graph Homomorphisms Based On Particular Total Colorings of Graphs  and Graphic Lattices}}

\vskip 0.4cm

{\Large Bing \textsc{Yao}$^{1}$,~Hongyu \textsc{Wang}$^{2,\dagger}$}\\

\vskip 0.4cm

{\large 1. College of Mathematics and Statistics, Northwest Normal University, Lanzhou, 730070 CHINA\\
2. School of Electronics Engineering and Computer Science, Peking University, Beijing, 100871, CHINA\\
$^{\dag}$ Corresponding author Hongyu Wang's email: why5126@pku.edu.cn}
\end{center}

\vskip 1.2cm

\begin{abstract}
Lattice-based cryptography is not only for thwarting future quantum computers, and is also the basis of Fully Homomorphic Encryption. Motivated from the advantage of graph homomorphisms we combine graph homomorphisms with graph total colorings together for designing new types of graph homomorphisms: totally-colored graph homomorphisms, graphic-lattice homomorphisms from sets to sets, every-zero graphic group homomorphisms from sets to sets. Our graph-homomorphism lattices are made up by graph homomorphisms. These new homomorphisms induce some problems of graph theory, for example, Number String Decomposition and Graph Homomorphism Problem.

\end{abstract}

\section{Introduction and preliminary}

A new security method built on an underlying architecture known as \emph{lattice-based cryptography} hides data inside complex mathematical problems. Lattice-based cryptography is not only for thwarting future quantum computers, and is also the basis of another encryption technology, called \emph{Fully Homomorphic Encryption}, which could make it possible to perform calculations on a file without ever seeing sensitive data or exposing it to hackers (Ref. \cite{Gena-Hahn-Claude-Tardif-1997, Pavol-Hell-Cambridge-2003, Alexander-Engstrom-Patrik-Noren-2011}).

\subsection{Graph homomorphisms in Homomorphic Encryption}

Homomorphisms provide a way of simplifying the structure of objects one wishes to study while preserving much of it that is of significance. It is not surprising that homomorphisms also appeared in graph theory, and that they have proven useful in many areas (Ref. \cite{Gena-Hahn-Claude-Tardif-1997, Alexander-Engstrom-Patrik-Noren-2011}). Graph homomorphisms have a great deal of applications in graph theory, computer science and other fields. The connection between locally constrained graph homomorphisms and degree matrices arising from an equitable partition of a graph have been explored in \cite{Fiala-Paulusma-Telle-2008}.

A main computational issue is: for every graph $H$ classifying the decision problem whether an input graph $G$ has a homomorphism of given type to the fixed graph $H$ as either NP-complete or polynomially solvable (Ref. \cite{Pavol-Hell-Cambridge-2003}). The comprehensive survey by Zhu \cite{Zhu-X-Discrete-Math2001} contains many other intriguing problems about graph homomorphism.

We, in this article, try to provide some design of graph homomorphisms which are based on topological structures and graph colorings. Topsnut-gpws (Ref. \cite{Wang-Xu-Yao-2016, Wang-Xu-Yao-Key-models-Lock-models-2016}) will play main roles in our homomorphisms (six colored graphs (a)-(f) shown in Fig.\ref{fig:example-00} are the Topsnut-gpws), because Topsnut-gpws are made up of two kinds of mathematical objects: topological structure and algebraic relation, such that attacker switch back and forth in two different languages, and are unable to convey useful information.

\subsection{Definition for graph homomorphisms}

We will use the standard notation and terminology of graph theory in this paper. Graphs will be simple, loopless and finite. A $(p,q)$-graph is a graph having $p$ vertices and $q$ edges. The \emph{cardinality} of a set $X$ is denoted as $|X|$, so the \emph{degree} of a vertex $x$ in a $(p,q)$-graph $G$ is written as $\textrm{deg}_G(x)=|N(x)|$, where $N(x)$ is the set of neighbors of the vertex $x$. A vertex $y$ is called a \emph{leaf} if $\textrm{deg}_G(y)=1$. A symbol $[a,b]$ stands for an integer set $\{a,a+1,\dots, b\}$ with two integers $a,b$ subject to $a<b$. All non-negative integers are collected in the set $Z^0$. A graph $G$ admits a \emph{labelling} $f:V(G)\rightarrow [a,b]$ means that $f(x)\neq f(y)$ for any pair of distinct vertices $x,y\in V(G)$ and, admits a \emph{coloring} $g:V(G)\rightarrow [a,b]$ means that $g(u)= g(v)$ for some two distinct vertices $u,v\in V(G)$. For a mapping $f:S\subset V(G)\cup E(G)\rightarrow [1,M]$, write color set by $f(S)=\{f(w):w\in S\}$. The definition of a graph homomorphism is shown as follows:

\begin{defn}\label{defn:definition-graph-homomorphism}
\cite{Bondy-2008} A \emph{graph homomorphism} $G\rightarrow H$ from a graph $G$ into another graph $H$ is a mapping $f: V(G) \rightarrow V(H)$ such that $f(u)f(v)\in E(H)$ for each edge $uv \in E(G)$. (see examples shown in Fig.\ref{fig:1-homomorphism}.)\qqed
\end{defn}

\begin{figure}[h]
\centering
\includegraphics[width=8cm]{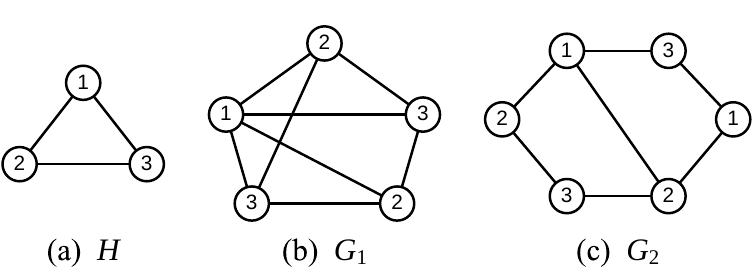}
\caption{\label{fig:1-homomorphism}{\small Two graph homomorphisms $G_i\rightarrow H$ based on $\theta_i: G_i\rightarrow H$ for $i=1,2$.}}
\end{figure}

\begin{rem} \label{rem:graph-homomorphism-11}
By \cite{Pavol-Hell-Cambridge-2003}, we have the following concepts:
\begin{asparaenum}[(a) ]
\item A homomorphism from a graph $G$ to itself is called an \emph{endomorphism}. An isomorphism from $G$ to $H$ is a particularly graph homomorphism from $G$ to $H$, also, they are \emph{homomorphically equivalent}.
\item Two graphs are homomorphically equivalent if each admits a homomorphism to the other, denoted as $G\leftrightarrow H$ which contains a homomorphism $G\rightarrow H$ from $G$ to $H$, and another homomorphism $H\rightarrow G$ from $H$ to $G$.
\item A homomorphism to the complete graph $K_n$ is exactly an $n$-coloring, so a homomorphism of $G$ to $H$ is also called an \emph{$H$-coloring} of $G$. The \textbf{homomorphism problem} for a fixed graph $H$, also called the $H$-coloring problem, asks whether or not an input graph $G$ admits a homomorphism to $H$.
\item By analogy with classical colorings, we associate with each $H$-coloring $f$ of $G$ a partition of $V (G)$ into the sets $S_h = f^{-1}(h)$, $h \in V (H)$. It is clear that a mapping $f: V (G) \rightarrow V (H)$ is a homomorphism of $G$ to $H$ if and only if the associated partition satisfies the following two constraints:

\quad (a-1) if $hh$ is not a loop in $H$, then the set $S_h$ is independent in $G$; and

\quad (a-2) if $hh'$ is not an edge (arc) of $H$, then there are no edges (arcs) from the set $S_h$ to the set $S_{h'}$ in $G$.

\quad Thus for a graph $G$ to admit an \emph{$H$-coloring is equivalent} to admitting a partition satisfying (a-1) and (a-2).
\item If $H$,$H'$ are homomorphically equivalent, then a graph $G$ is $H$-colorable if and only if it is $H'$-colorable.
\item Suppose that $H$ is a subgraph of $G$. We say that $G$ retracts to $H$, if there exists a homomorphism $f:G\rightarrow H$, called a \emph{retraction}, such that $f(u)=u$ for any vertex of $H$. A \emph{core} is a graph which does not retract to a proper subgraph. Any graph is homomorphically equivalent to a core.\qqed
\end{asparaenum}
\end{rem}

\section{Results on graph homomorphisms}

\subsection{Graph homomorphisms of uncolored graphs}

By Definition \ref{defn:definition-graph-homomorphism}, we have a result as follows:

\begin{prop}\label{thm:observation}
Suppose that $\varphi: G \rightarrow H$ is a graph homomorphism. Then $\varphi$ is an
isomorphism if and only if $\varphi$ is bijective and also a homomorphism. In particular,
if $G = H$ then $\varphi$ is an automorphism if and only if it is bijective.
\end{prop}

\begin{defn}\label{defn:faithful}
\cite{Gena-Hahn-Claude-Tardif-1997} A graph homomorphism $\varphi: G \rightarrow H$ is called \emph{faithful} if $\varphi(G)$ is an induced subgraph of $H$, and called \emph{full} if $uv\in E(G)$ if and only if $\varphi(u)\varphi(v) \in E(H)$.
\end{defn}

A graph homomorphism $\varphi: G \rightarrow H$ is faithful when there is an edge between any two pre-images (inverse image) $\varphi^{-1}(x)$ and $\varphi^{-1}(y)$ such that $xy$ is an edge of $H$, $\varphi^{-1}(x)\cup \varphi^{-1}(y)$ induces a complete bipartite graph whenever $xy\in E(H)$. Moreover, $\varphi^{-1}(u)\varphi^{-1}(v)$ is an edge in $G$ if and only if $uv$ is an edge in $H$, thus
\begin{thm}\label{thm:bijective-graph-homomorphism}
\cite{Gena-Hahn-Claude-Tardif-1997} A faithful bijective graph homomorphism is an isomorphism, that is $G \cong H$.
\end{thm}

\begin{thm}\label{thm:sequence-graph-homomorphisms}
There are infinite graphs $G^*_n$ forming a sequence $\{G^*_n\}$, such that $G^*_n\rightarrow G^*_{n-1}$ is really a graph homomorphism for $n\geq 1$.
\end{thm}
\begin{proof} First, we present an algorithm as follows: $G_0$ is a triangle $\Delta x_1x_2x_3$, we use a coloring $h$ to color the vertices of $G_0$ as $h(x_i)=0$ with $i\in [1,3]$.

Step 1: Add a new $y$ vertex for each edge $x_ix_j$ of the triangle $\Delta x_1x_2x_3$ with $i\neq j$, and join $y$ with two vertices $x_i$ and $x_j$ of the edge $x_ix_j$ by two new edges $yx_i$ and $yx_j$, the resulting graph is denoted by $G_1$, and color $y$ with $h(y)=1$.

Step 2: Add a new $w$ vertex for each edge $uv$ of $G_1$ if $h(u)=1$ and $h(v)=0$ (or $h(v)=1$ and $h(u)=0$), and join $y$ respectively with two vertices $u$ and $v$ by two new edges $wu$ and $wv$, the resulting graph is denoted by $G_2$, and color $w$ with $h(w)=2$.

Step $n$: Add a new $\gamma$ vertex for each edge $\alpha\beta$ of $G_{n-1}$ if $h(\alpha)=n-1$ and $h(\beta)=n-2$ (or $h(\alpha)=n-2$ and $h(\beta)=n-1$), and join $\gamma$ respectively with two vertices $\alpha$ and $\beta$ by two new edges $\gamma\alpha$ and $\gamma\beta$, the resulting graph is denoted by $G_n$, and color $\gamma$ with $h(\gamma)=n$.

Second, we construct a graph $G^*_n=G_n\cup K_1$ with $n\geq 0$, where $K_1$ is a complete graph of one vertex $z_0$. For each $n\geq 1$, there is a mapping $\theta_n:V(G^*_n)\rightarrow V(G^*_{n-1})$ in this way: $V(G^*_n\setminus V^n_{(2)})=V(G^*_{n-1}\setminus {V(K_1)})$, each $x\in V^n_{(2)}$ holds $\theta_n(x)=z_0$, where $V^n_{(2)}$ is the set of vertices of degree two in $G^*_n$. So $G^*_n\rightarrow G^*_{n-1}$ is really a graph homomorphism. We write this case by $\{G^*_n\}\rightarrow G^*_0$, called as a \emph{graph homomorphism sequence}.
\end{proof}

The notation $\{G^*_n\}\rightarrow G^*_0$ can be written as
\begin{equation}\label{eqa:inverse-limit}
\lim_{\infty \rightarrow 0}\{G^*\}^{\infty}_{0}=G^*_0
\end{equation} called an \emph{inverse limitation}. There are many graph homomorphism sequence $\{G^*_n\}$ holding $G^*_n\rightarrow G^*_{n-1}$, i.e., $\{G^*_n\}\rightarrow G^*_0$ in network science. For example, we can substitute the triangle $G_0$ in the proof of Theorem \ref{thm:sequence-graph-homomorphisms} by any connected graph.

\subsection{Totally-colored graph homomorphisms}

We propose a new type of graph homomorphisms by combining graph homomorphisms and graph total colorings together as follows:

\begin{defn}\label{defn:gracefully-graph-homomorphism}
Let $G\rightarrow H$ be a graph homomorphism from a $(p,q)$-graph $G$ to another $(p',q')$-graph $H$ based on a mapping $\alpha: V(G) \rightarrow V(H)$ such that $\alpha(u)\alpha(v)\in E(H)$ for each edge $uv \in E(G)$. The graph $G$ admits a total coloring $f$, the graph $H$ admits a total coloring $g$. Write $f(E(G))=\{f(uv):uv \in E(G)\}$ and $g(E(H))=\{g(\alpha(u)\alpha(v)):\alpha(u)\alpha(v)\in E(H)\}$, there are the following conditions:
\begin{asparaenum}[(\textrm{C}-1) ]
\item \label{bipartite} $V(G)=X\cup Y$, each edge $uv \in E(G)$ holds $u\in X$ and $v\in Y$ true. $V(H)=W\cup Z$, each edge $\alpha(u)\alpha(v)\in E(G)$ holds $\alpha(u)\in W$ and $\alpha(v)\in Z$ true.
\item \label{edge-difference} $f(uv)=|f(u)-f(v)|$ for each $uv \in E(G)$, $g(\alpha(u)\alpha(v))=|g(\alpha(u))-g(\alpha(v))|$ for each $\alpha(u)\alpha(v)\in E(H)$.
\item \label{edge-homomorphism} $f(uv)=g(\alpha(u)\alpha(v))$ for each $uv \in E(G)$.
\item \label{vertex-color-set} $f(x)\in [1,q+1]$ for $x\in V(G)$, $g(y)\in [1,q'+1]$ with $y\in V(H)$.
\item \label{odd-vertex-color-set} $f(x)\in [1,2q+2]$ for $x\in V(G)$, $g(y)\in [1,2q'+2]$ with $y\in V(H)$.
\item \label{grace-color-set} $[1,q]=f(E(G))=g(E(H))=[1,q']$.
\item \label{odd-grace-color-set} $[1,2q-1]=f(E(G))=g(E(H))=[1,2q'-1]$.
\item \label{set-ordered} Set-ordered property: $\max f(X)<\min f(Y)$ and $\max g(W)<\min g(Z)$.
\end{asparaenum}

We say $G\rightarrow H$ to be: (i) a \emph{bipartitely graph homomorphism} if (C-\ref{bipartite}) holds true; (ii) a \emph{gracefully graph homomorphism} if (C-\ref{edge-difference}), (C-\ref{edge-homomorphism}), (C-\ref{vertex-color-set}) and (C-\ref{grace-color-set}) hold true; (ii) a \emph{set-ordered gracefully graph homomorphism} if (C-\ref{edge-difference}), (C-\ref{edge-homomorphism}), (C-\ref{vertex-color-set}), (C-\ref{grace-color-set}) and (C-\ref{set-ordered}) hold true; (iv) an \emph{odd-gracefully graph homomorphism} if (C-\ref{edge-difference}), (C-\ref{edge-homomorphism}), (C-\ref{odd-vertex-color-set}) and (C-\ref{odd-grace-color-set}) hold true; (v) a \emph{set-ordered odd-gracefully graph homomorphism} if (C-\ref{edge-difference}), (C-\ref{edge-homomorphism}), (C-\ref{odd-vertex-color-set}), (C-\ref{odd-grace-color-set}) and (C-\ref{set-ordered}) hold true.\qqed
\end{defn}

There are five graph homomorphisms $G_{(k)}\rightarrow G_{(\textrm{a})}$ for $k=$b,c,d,e,f in Fig.\ref{fig:example-00}, each $G_{(k)}\rightarrow G_{(\textrm{a})}$ is a set-ordered gracefully graph homomorphism. However, any two of six uncolored graphs $G_{(m)}$ with $m=$aa,bb,cc,dd,ee,ff are not isomorphic from each other.

\begin{figure}[h]
\centering
\includegraphics[width=8cm]{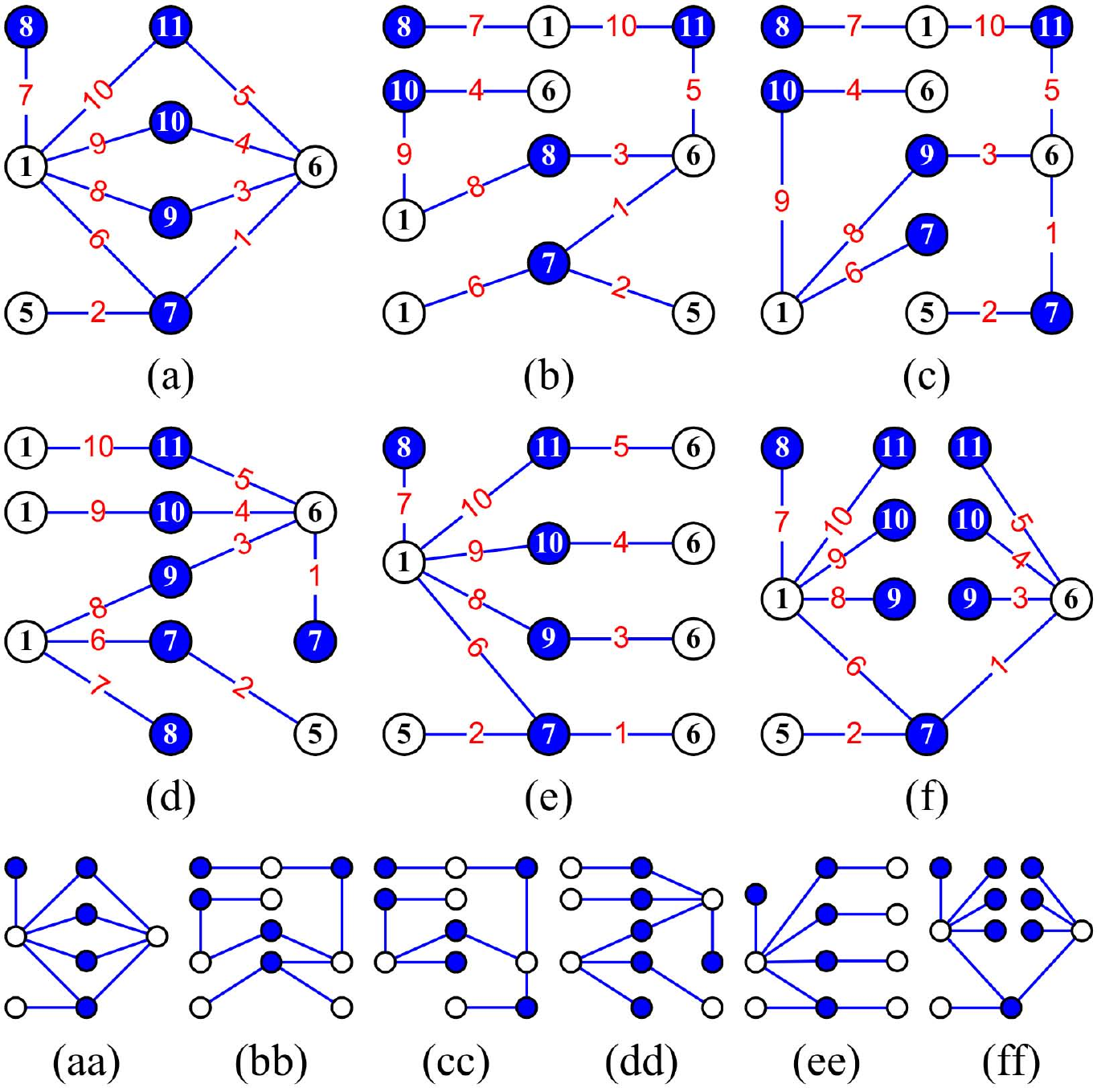}
\caption{\label{fig:example-00}{\small Five set-ordered gracefully graph homomorphisms $G_{(k)}\rightarrow G_{(\textrm{a})}$ for $k=$b,c,d,e,f; and six uncolored graphs $G_{(m)}$ with $m=$aa,bb,cc,dd,ee,ff.}}
\end{figure}

\begin{rem} \label{rem:gracefully-graph-homomorphism}
\begin{asparaenum}[(i) ]
There are the following issues about Definition \ref{defn:gracefully-graph-homomorphism}:
\item A gracefully graph homomorphism $G\rightarrow H$ holds $|f(E(G))|=|g(E(H))|$ with $q=q'$ and $|f(V(G))|\geq |g(V(H))|$, in general. The graph $G$ admits a \emph{set-ordered graceful total coloring} $f$ and the graph $H$ admits a \emph{set-ordered gracefully total coloring} $g$ by the definitions of topological coding. We call the totally-colored graph homomorphisms defined in Definition \ref{defn:gracefully-graph-homomorphism} as \emph{$W$-type totally-colored graph homomorphisms}, where a ``$W$-type totally-colored graph homomorphism'' is one of the totally-colored graph homomorphisms; and we say that the graph $G$ admits a $W$-type totally-colored graph homomorphism to $H$ in a $W$-type totally-colored graph homomorphism $G\rightarrow H$.
\item If $T\rightarrow H$ is a $W$-type totally-colored graph homomorphism, and so is $H\rightarrow T$, we say $T$ and $H$ are \emph{homomorphically equivalent} from each other, denoted as $T\leftrightarrow H$. If two graphs $G$ admitting a coloring $f$ and $H$ admitting a coloring $g$ hold $f(x)=g(\varphi(x))$ and $G\cong H$, we write this case by $G=H$ in the following discussion.\qqed
\end{asparaenum}
\end{rem}

We can see three gracefully graph homomorphisms $\varphi_i:V(H_i)\rightarrow V(H_{ii})$ for $i=1,2,3$ in Fig.\ref{fig:gracefully-homomorphism-1}. However, $H_i\neq H_j$ since their colorings $g_i(x)\neq g_j(x)$ with $1\leq i,j\leq 3$, although $H_i\cong H_j$; and moreover $H_{ii}\neq H_{jj}$ because of their colorings $h_{ii}(x)\neq h_{jj}(x)$ with $1\leq i,j\leq 3$. By the way, we have three graph homomorphisms $\theta_i:V(H_{ii})\rightarrow V(H)$ for $i=1,2,3$. Observe the inverse $\theta^{-1}_1$ of the graph homomorphism $\theta_1$, the vertex $x_{11}\in V(H_{11})$ produces a set $\theta^{-1}_1(x_{11})=\{x_{1},x_{2}\}\subset V(H_1)$, and the vertex $y_{11}\in V(H_{11})$ produces a set $\theta^{-1}_1(y_{11})=\{y_{1},y_{2}\}\subset V(H_1)$. Thereby, the graph homomorphism $\theta_1$ is not full (bijective), so are two graph homomorphisms $\theta_2,\theta_3$.

\begin{figure}[h]
\centering
\includegraphics[width=8cm]{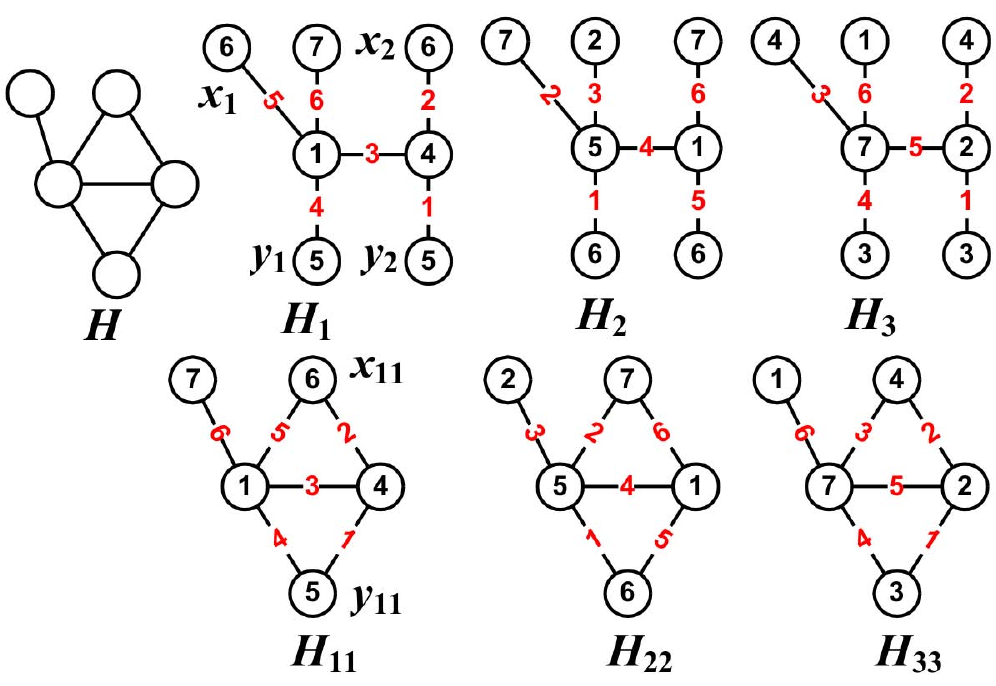}
\caption{\label{fig:gracefully-homomorphism-1}{\small For $i\neq j$, $H_i\neq H_j$ although $H_i\cong H_j$, and $H_{ii}\neq H_{jj}$ although $H_{ii}\cong H_{jj}$.}}
\end{figure}

In Fig.\ref{fig:gracefully-homomorphism}, we can see $G_i\cong G_j$ for $i\neq j$, however, $G_i\neq G_j$ since their colorings $f_i(x)\neq f_j(x)$ with $1\leq i,j\leq 4$, although $G_i\cong G_j$. Suppose that there is a graph $G^*_1$ such that we have a set-ordered gracefully graph homomorphism $\pi:V(G_1)\rightarrow V(G^*_1)$, then $G_1$ admits a set-ordered graceful total coloring $f_1$ and $G^*_1$ admits a set-ordered gracefully total coloring $g_1$. So, the vertex color sets $f_1(V(G_1))=[1,7]=g_1(V(G^*_1))$, and the edge color sets $f_1(E(G_1))=[1,6]=g_1(E(G^*_1))$. The inverse $\pi^{-1}(x)$ for $x\in V(G^*_1)$ corresponds a vertex $x'\in V(G_1)$ but a subset of $V(G_1)$, which means $\pi^{-1}(x)\pi^{-1}(y)$ is an edge of $G_1$ if and only if $xy$ is an edge of $G^*_1$, immediately, $G_1 \cong G^*_1$, and $\pi$ is bijective. Since $|f_1(u)-f_1(v)|=|g_1(\pi(u))-g_1(\pi(v))|$ by Definition \ref{defn:gracefully-graph-homomorphism}, and $uv\leftrightarrow \pi(u)\pi(v)$ is unique, we claim that $G_1=G^*_1$.

\begin{thm}\label{thm:parti-gracecular-bijective}
If a (set-ordered) gracefully graph homomorphism $\varphi: G \rightarrow H$ defined in Definition \ref{defn:gracefully-graph-homomorphism} holds $f(V(G))=g(V(H))$ and $f(E(G))=g(E(H))$, then $G=H$.
\end{thm}

\begin{figure}[h]
\centering
\includegraphics[width=8cm]{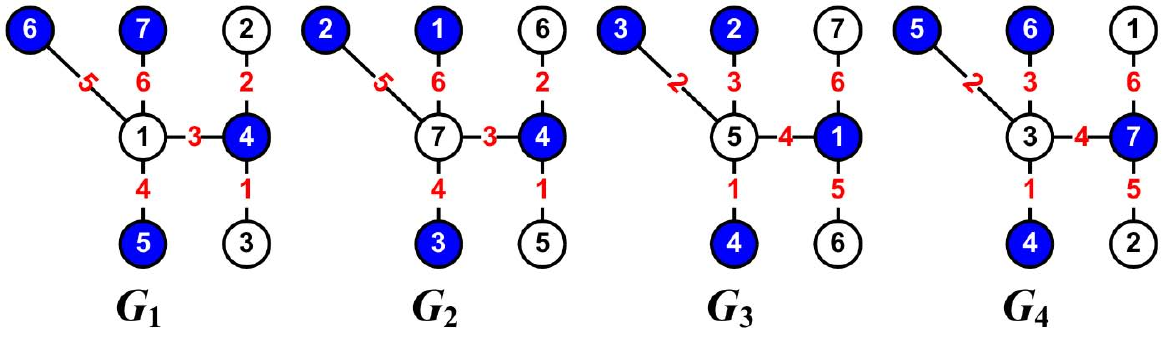}
\caption{\label{fig:gracefully-homomorphism}{\small $G_1$ and $G_2$ are a pair of mutually dual set-ordered graceful colorings, so are $G_3$ and $G_4$.}}
\end{figure}

By Definition \ref{defn:faithful} and Theorem \ref{thm:bijective-graph-homomorphism}, we have

\begin{thm}\label{thm:set-ordered-gracefully-bijective}
If a (set-ordered) gracefully graph homomorphism $\varphi: G \rightarrow H$ is faithful bijective, then $G=H$.
\end{thm}

A connected graph $G$ admits a set-ordered gracefully total coloring $f$, so $V(G)=X\cup Y$ and each edge $uv\in E(G)$ holds $u\in X$ and $v\in Y$ true. Since the property of a set-ordered gracefully total coloring $f$ is $\max f(X)<\min f(Y)$, without loss of generality, we can have $1=f(u_1)\leq f(u_2)\leq \cdots \leq f(u_a) <f(v_1)\leq f(v_2)\leq \cdots \leq f(v_b)=q+1$, $f(u_iv_j)=f(v_j)-f(u_i)$, where $X=\{u_1,u_2,\dots, u_a\}$ and $Y=\{v_1,v_2,\dots, v_b\}$, $a+b=|V(G)|$ and $q=|E(G)|$. We come to set another total coloring $g$ of $G$ in the way: $g(u_i)=2f(u_i)$ for $u_i\in X$ and $g(v_j)=2f(v_j)-1$ for $v_j\in Y$, as well as $g(u_iv_j)=g(v_j)-g(u_i)$. Clearly, $g(E(G))=\{1,3,5,\dots ,2q+1\}$ from $f(E(G))=[1,q]$, which means that $g$ is really a set-ordered odd-gracefully total coloring of $G$. We say: \emph{A connected graph admits a set-ordered gracefully total coloring if and only if it admits a set-ordered odd-gracefully total coloring}. We can obtain the results based on odd-gracefully graph homomorphisms, like that in Theorem \ref{thm:parti-gracecular-bijective} and Theorem \ref{thm:set-ordered-gracefully-bijective}.

\section{Graph homomorphisms with lattices, every-zero graphic groups}

The authors in \cite{Wang-Yao-Star-type-lattices-2020} and \cite{Yao-Wang-Su-Sun-ITOEC2020} introduce graphic lattices, and give connections between traditional lattices and graphic lattices. In this subsection, we will introduce: (i) graph-homomorphism lattices based on graph homomorphisms; (ii) every-zero graphic group homomorphisms; and (iii) graphic-lattice homomorphisms from sets to sets.

\subsection{Graph homomorphism lattices} Graph homomorphism lattices are like graphic lattices. Let $\textrm{H}_{\textrm{om}}(H,W)$ be the set of all $W$-type totally-colored graph homomorphisms $G\rightarrow H$. For a fixed $W_k$-type graph homomorphism, suppose that there are mutually different total colorings $g_{k,1},g_{k,2},\dots, g_{k,m_k}$ of the graph $H$ to form $W_k$-type graph homomorphisms $G\rightarrow H_{k,i}$ with $i\in [1,m_k]$, where $H_{k,i}$ is a copy of $H$ and colored by a total coloring $g_{k,i}$. Thereby, we have the sets $\textrm{H}_{\textrm{om}}(H_{k,i},W_k)$ with $i\in [1,m_k]$, and $\textrm{H}_{\textrm{om}}(H_{k},W_k)=\bigcup^{m_k}_{i=1}\textrm{H}_{\textrm{om}}(H_{k,i},W_k)$. We get a \emph{$W_k$-type totally-colored graph homomorphism lattice} as follows
\begin{equation}\label{eqa:Wk-type-graph homomorphism-lattice}
{
\begin{split}
\textbf{\textrm{L}}(\textbf{\textrm{W}}_k,\textbf{\textrm{H}}^k_{\textrm{om}})=& \Biggr \{\bigcup^{m_k}_{i=1}a_i(G\rightarrow H_{k,i}):a_k\in \{0,1\};\\
&H_{k,i}\in \textrm{H}_{\textrm{om}}(H_{k},W_k)\Biggr \}
\end{split}}
\end{equation}
with $\sum^{m_k}_{i=1}a_i=1$ and the base $\textbf{\textrm{H}}^k_{\textrm{om}}=(H_{k,i})^{m_k}_{i=1}$.

For example, as a $W_k$-type totally-colored graph homomorphism is a set-ordered gracefully graph homomorphism, $G$ admits a set-ordered graceful total coloring $f$ and $H$ admits a set-ordered gracefully total coloring $g_k$ in a set-ordered gracefully graph homomorphism $G\rightarrow H_k$. Thereby, a $W_k$-type totally-colored graph homomorphism lattice may be feasible and effective in application. In real computation, finding all of mutually different set-ordered gracefully total colorings of the graph $H$ is a difficult math problem, since there is no polynomial algorithm for this problem.

Notice that $\textrm{H}_{\textrm{om}}(H,W)=\bigcup^M_{k=1}\bigcup^{m_k}_{i=1}\textrm{H}_{\textrm{om}}(H_{k,i},W_k)$, where $M$ is the number of all $W$-type totally-colored graph homomorphisms, immediately, we get a \emph{$W$-type totally-colored graph homomorphism lattice}
\begin{equation}\label{eqa:total-type-graph-homomorphism-lattice}
{
\begin{split}
\textbf{\textrm{L}}(\textbf{\textrm{W}},\textbf{\textrm{H}}_{\textrm{om}})=\bigcup ^M_{k=1}\textbf{\textrm{L}}(\textbf{\textrm{W}}_k,\textbf{\textrm{H}}^k_{\textrm{om}})
\end{split}}
\end{equation}
with the base $\textbf{\textrm{H}}_{\textrm{om}}=((H_{k,i})^{m_k}_{i=1})^M_{k=1}$.

\subsection{Every-zero graphic group homomorphisms}
Every-zero graphic groups have been introduced and discussed in \cite{Wang-Su-Sun-Yao-submitted-ITOEC2020, Yao-Zhang-Sun-Mu-Sun-Wang-Wang-Ma-Su-Yang-Yang-Zhang-2018arXiv, Yao-Mu-Sun-Sun-Zhang-Wang-Su-Zhang-Yang-Zhao-Wang-Ma-Yao-Yang-Xie2019}.
Let $F_f(G)=\{G_i:i\in [1,m]\}$ be a set of graphs, where $G_i\cong G$ with $G_1=G$, and $G$ admits a $W$-type total coloring $f$, each $G_i$ admits a $W$-type total coloring $f_i$ induced by $f$ holding $f_i(xy)\,(\textrm{mod}\,M)=f(xy)$ for $xy\in E(G)$, where $M$ is a constant. We select any $G_k$ as the \emph{zero} for the operation ``$G_i\oplus G_j$'' on the graph set $F_f(G)$, and define
\begin{equation}\label{eqa:graphic-group-definition}
f_i(x)+f_j(x)-f_k(x)=f_{\lambda}(x)
\end{equation} with $\lambda=i+j-k\,(\textrm{mod}\,M)$ for each vertex $x\in V(G)$, and $f_{\lambda}(xy)\,(\textrm{mod}\,M)=f(xy)$ for each edge $xy\in E(G)$. By the operation ``$G_i\oplus G_j$'' defined in (\ref{eqa:graphic-group-definition}), it is not hard to verify $G_i\oplus G_k=G_i$, $G_i\oplus G_j=G_j\oplus G_i$, $(G_i\oplus G_j)\oplus G_s=G_i\oplus (G_j\oplus G_s)$. So, we get an \emph{every-zero graphic group} $\{F_f(G);\oplus\}$ introduced in \cite{Yao-Zhang-Sun-Mu-Sun-Wang-Wang-Ma-Su-Yang-Yang-Zhang-2018arXiv}.

Let $\{F_h(H);\oplus\}$ be another \emph{every-zero graphic group} under the operation ``$H_i\oplus H_j$'' defined in (\ref{eqa:graphic-group-definition-00}), where the graph set $F_h(H)=\{H_i:i\in [1,m]\}$, $H_i\cong H$ with $H_1=H$, and $H$ admits a $W$-type total coloring $h$, each $H_i$ admits a $W$-type total coloring $h_i$ induced by $h$ holding $h_i(xy)\,(\textrm{mod}\,M)=h(xy)$ for $xy\in E(G)$, where $M$ is a constant. For the zero selected arbitrarily from the graph set $F_h(H)$, we have the operation ``$H_i\oplus H_j$'' defined as follows
\begin{equation}\label{eqa:graphic-group-definition-00}
h_i(w)+h_j(w)-h_k(x)=h_{\mu}(w)
\end{equation} with $\mu=i+j-k\,(\textrm{mod}\,M)$ for each vertex $w\in V(H)$, and $h_{\mu}(wz)\,(\textrm{mod}\,M)=h(wz)$ for each edge $wz\in E(H)$.

Suppose that there are graph homomorphisms $G_i\rightarrow H_i$ defined by $\theta_i:V(G_i)\rightarrow V(H_i)$ with $i\in [1,m]$. We define $\theta=\bigcup^m_{i=1}\theta_i$, and have an \emph{every-zero graphic group homomorphism} $\{F_f(G);\oplus\}\rightarrow \{F_h(H);\oplus\}$ from a set $F_f(G)$ to another set $F_h(H)$.

Two sets $F_f(G)=\{G_i:i\in [1,7]\}$ and $F_h(H)=\{H_i:i\in [1,7]\}$ shown in Fig.\ref{fig:graphic-group} distribute us seven graph homomorphisms $\theta_j:V(G_j)\rightarrow V(H_j)$ with $j\in [1,7]$. It is not hard to verify two every-zero graphic groups $F_f(G)=\{G_i:i\in [1,7]\}$ and $F_h(H)=\{H_i:i\in [1,7]\}$ by the formulae (\ref{eqa:graphic-group-definition}) and (\ref{eqa:graphic-group-definition-00}).

\begin{figure}[h]
\centering
\includegraphics[width=8cm]{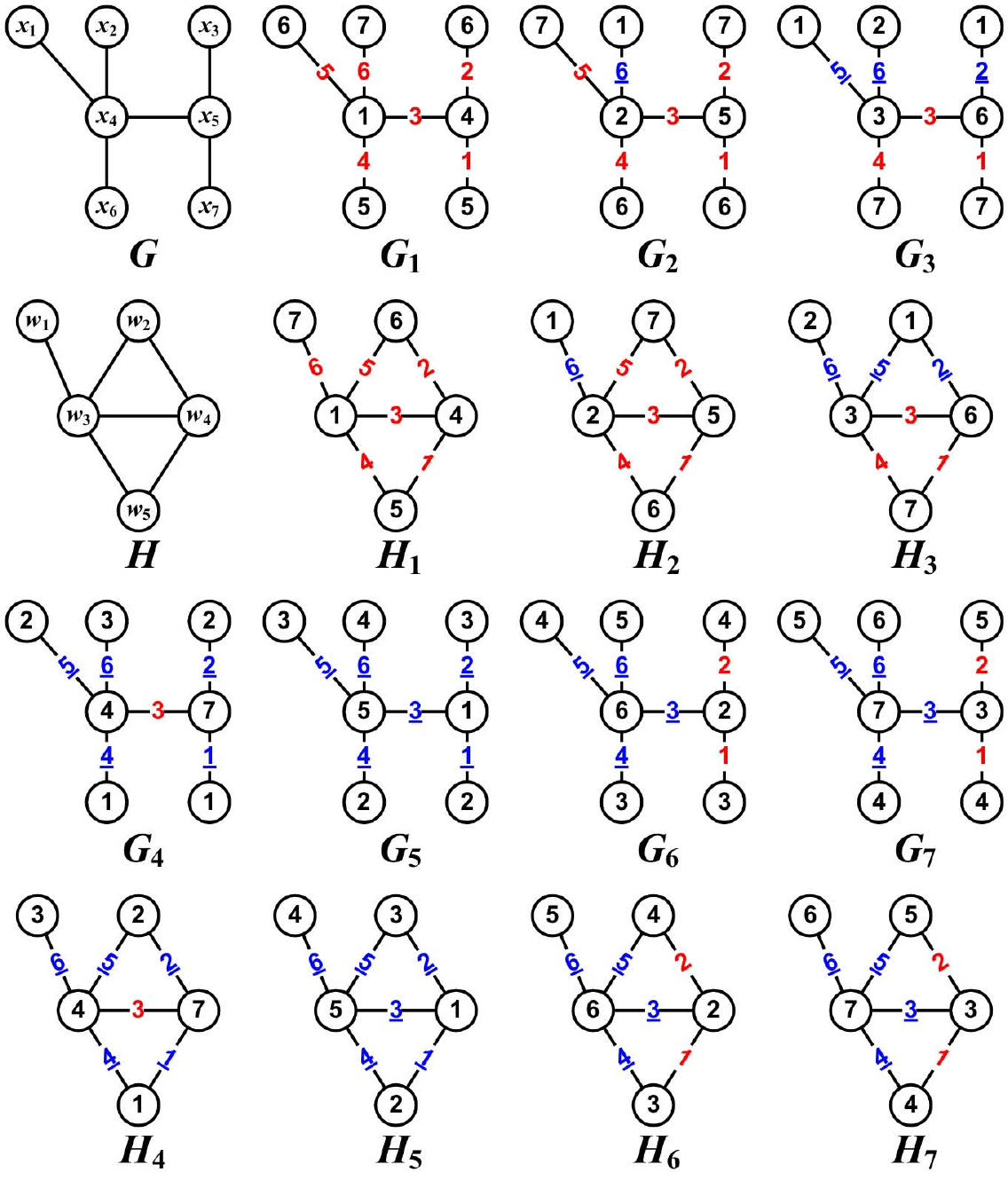}
\caption{\label{fig:graphic-group}{\small Two every-zero graphic groups $\{F_f(G);\oplus\}$ and $\{F_h(H);\oplus\}$.}}
\end{figure}

\subsection{Graphic lattice homomorphisms}
Let $\textrm{\textbf{G}}=(G_k)^m_{k=1}$ and $\textrm{\textbf{H}}=(H_k)^m_{k=1}$ be two \emph{bases}, and let $\theta_k:V(G_k)\rightarrow V(H_k)$ be a $W_k$-type totally-colored graph homomorphism with $k\in [1,m]$, and let $(\bullet)$ be a graph operation on graphs. Suppose that $F$ and $J$ are two sets of graphs, such that each graph $G\in F$ corresponds a graph $H\in J$, and there is a $W_k$-type totally-colored graph homomorphism $\theta_{G,H}:V(G)\rightarrow V(H)$. We have the following graphic lattices:
\begin{equation}\label{eqa:general-operation}
{
\begin{split}
&\textbf{\textrm{L}}(\textbf{\textrm{F}}(\bullet)\textbf{\textrm{G}})=\left \{(\bullet)^m_{i=1}a_iG_i:~a_i\in Z^0;G_i\in \textrm{\textbf{G}}\right \}\\
&\textbf{\textrm{L}}(\textbf{\textrm{J}}(\bullet)\textbf{\textrm{H}})=\left \{(\bullet)^m_{j=1}b_jH_j:~b_j\in Z^0;H_j\in \textrm{\textbf{H}}\right \}
\end{split}}
\end{equation} with $\sum^m_{i=1}a_i\geq 1$ and $\sum^m_{j=1}b_j\geq 1$. Let $\pi=(\bigcup^m_{k=1} \theta_k)\cup (\bigcup _{G\in F,H\in J}\theta_{G,H})$. We have a \emph{$W$-type graphic lattice homomorphism}
\begin{equation}\label{eqa:general-graphic lattice-homomorphism}
\pi:\textbf{\textrm{L}}(\textbf{\textrm{F}}(\bullet)\textbf{\textrm{G}})\rightarrow \textbf{\textrm{L}}(\textbf{\textrm{J}}(\bullet)\textbf{\textrm{H}}).
\end{equation}

In particular cases, we have: (1) The operation $(\bullet)=\ominus$ is an operation by joining some vertices $x_{k,i}$ of $G_i$ with some vertices $y_{k,j}$ of $G_j$ together by new edges $x_{k,i}y_{k,j}$ with $k\in [1,a_k]$ and $a_k\geq 1$, the resultant graph is denoted as $G_i\ominus G_j$, called \emph{edge-joined graph}. (2) The operation $(\bullet)=\odot$ is an operation by coinciding a vertex $u_{k,i}$ of $G_i$ with some vertex $v_{k,j}$ of $G_j$ into one vertex $u_{k,i}\odot v_{k,j}$ for $k\in [1,b_k]$ with integer $b_k\geq 1$, the resultant graph is denoted as $G_i\odot G_j$.

Thereby, we get an edge-joined graph $H_i\ominus H_j$ since $\theta_k(x_{k,i})\in V(H_i)$ and $\theta_k(y_{k,j})\in V(H_j)$, $\theta_k(x_{k,i}y_{k,j})\in E(H_i\ominus H_j)$ and a $W_k$-type totally-colored graph homomorphism $\theta_k:V(G_i\ominus G_j)\rightarrow V(H_i\ominus H_j)$. Similarly, we have another $W_k$-type totally-colored graph homomorphism $\phi_k:V(G_i\odot G_j)\rightarrow V(H_i\odot H_j)$. In totally, we have two $W_k$-type totally-colored graph homomorphisms
\begin{equation}\label{eqa:c3xxxxx}
{
\begin{split}
&\theta:\ominus^m_{i=1} V(G_i)\rightarrow \ominus^m_{i=1} V(H_i)\\
&\phi:\odot^m_{i=1} V(G_i)\rightarrow \odot^m_{i=1} V(H_i).
\end{split}}
\end{equation}

We have two graphic lattices based on the operation ``$\ominus$'':
\begin{equation}\label{eqa:graphic lattice-homomorphism-11}
{
\begin{split}
&\textbf{\textrm{L}}(\ominus\textbf{\textrm{G}})=\left \{\ominus ^m_{k=1}a_kG_k:~a_k\in Z^0;G_k\in \textrm{\textbf{G}}\right \}\\
&\textbf{\textrm{L}}(\ominus\textbf{\textrm{H}})=\left \{\ominus ^m_{k=1}b_kH_k:~b_k\in Z^0;H_k\in \textrm{\textbf{H}}\right \}
\end{split}}
\end{equation}
with $\sum^m_{k=1}a_k\geq 1$ and $\sum^m_{k=1}b_k\geq 1$.

The above works enable us to get a homomorphism $\theta:\textbf{\textrm{L}}(\ominus\textbf{\textrm{G}})\rightarrow \textbf{\textrm{L}}(\ominus\textbf{\textrm{H}})$, called \emph{$W$-type graphic lattice homomorphism}. Similarly, we have another $W$-type graphic lattice homomorphism $\pi':\textbf{\textrm{L}}(\odot\textbf{\textrm{G}})\rightarrow \textbf{\textrm{L}}(\odot\textbf{\textrm{H}})$ by the following two graphic lattices based on the operation ``$\odot$''
\begin{equation}\label{eqa:graphic lattice-homomorphism-33}
\textbf{\textrm{L}}(\odot\textbf{\textrm{G}})=\left \{\odot ^m_{k=1}c_kG_k:~c_k\in Z^0;G_k\in \textrm{\textbf{G}}\right \}
\end{equation}
with $\sum^m_{k=1}c_k\geq 1$, and
\begin{equation}\label{eqa:graphic lattice-homomorphism-44}
\textbf{\textrm{L}}(\odot\textbf{\textrm{H}})=\left \{\odot ^m_{k=1}d_kH_k:~d_k\in Z^0;H_k\in \textrm{\textbf{H}}\right \}
\end{equation}
with $\sum^m_{k=1}d_k\geq 1$. Notice that there are mixed operations of the operation ``$\ominus$'' and the operation ``$\odot$'', so we have more complex $W$-type graphic lattice homomorphisms. If two \emph{bases} $\textbf{\textrm{G}}_{\textrm{group}}=\{F_f(G);\oplus\}$ and $\textbf{\textrm{H}}_{\textrm{group}}=\{F_h(H)$; $\oplus\}$ are two every-zero graphic groups, so we have an every-zero graphic group homomorphism $\varphi:\textbf{\textrm{G}}_{\textrm{group}}\rightarrow \textbf{\textrm{H}}_{\textrm{group}}$ and two \emph{every-zero graphic group homomorphisms}:
$$\textbf{\textrm{L}}(\ominus\textbf{\textrm{G}}_{\textrm{group}})\rightarrow \textbf{\textrm{L}}(\ominus\textbf{\textrm{H}}_{\textrm{group}}),\textbf{\textrm{L}}(\odot\textbf{\textrm{G}}_{\textrm{group}})\rightarrow \textbf{\textrm{L}}(\odot\textbf{\textrm{H}}_{\textrm{group}}).$$

\subsection{Authentications based on various graph homomorphisms}

In Fig.\ref{fig:example-00}, if we select a Topsnut-gpw (a) as a \emph{public key} $G_{(\textrm{a})}$ in Fig.\ref{fig:example-00}, then we have at least five Topsnut-gpws $G_{(k)}$ with $k=$b,c,d,e,f, as \emph{private keys}, to form five set-ordered gracefully graph homomorphisms $G_{(k)}\rightarrow G_{(\textrm{a})}$. In the topological structure of view, $G_{(ii)}\not\cong G_{(jj)}$ for $i,j=$b,c,d,e,f and $i\neq j$. We, by these six Topsnut-gpws, have a Topcode-matrix (Ref. \cite{Sun-Zhang-Zhao-Yao-2017, Yao-Sun-Zhao-Li-Yan-2017, Yao-Zhang-Sun-Mu-Sun-Wang-Wang-Ma-Su-Yang-Yang-Zhang-2018arXiv})
\begin{equation}\label{eqa:Topcode-matrix-vs-6-Topsnut-gpws}
\centering
{
\begin{split}
T_{code}= \left(
\begin{array}{ccccccccccc}
6&5&6&6&6&1&1&1&1&1\\
1&2&3&4&5&6&7&8&9&10\\
7&7&9&10&11&7&8&9&10&11
\end{array}
\right)
\end{split}}
\end{equation} So, each of these six Topsnut-gpws corresponds the Topcode-matrix $T_{code}$. Moreover, the Topcode-matrix $T_{code}$ can distribute us $30!$ \emph{number strings} $S_{k}$ with $k\in [1,30!]$ like the following number string
$$S_{1}=617725639104665117611678711891089111011$$
with $39$ bytes. For the reason of authentications, we have to solve a problem, called Number String Decomposition and Graph Homomorphism Problem (NSD-GHP), as follows:

\vskip 0.2cm

\textbf{NSD-GHP:}

Given a number string $S_{1}=c_1c_2\cdots c_m$ with $c_i\in [0,9]$, decompose it into $30$ segments $c_1c_2\cdots c_m=a_1a_2\cdots a_{30}$ with $a_j=c_{n_j}c_{n_j+1}\cdots c_{n_{j+1}}$ with $j\in [1,29]$, $n_1=1$ and $n_{30}=m$. And use $a_k$ with $k\in [1,30]$ to reform the Topcode-matrix $T_{code}$ in (\ref{eqa:Topcode-matrix-vs-6-Topsnut-gpws}), and moreover reconstruct all Topsnut-gpws (like six Topsnut-gpws corresponds shown in Fig.\ref{fig:example-00}). By the found Topsnut-gpws corresponding the common Topcode-matrix $T_{code}$, find the public Topsnut-gpws $H_i$ and the private Topsnut-gpws $G_i$ as we desired, such that each mapping $\varphi_i:V(G_i)\rightarrow V(H_i)$ forms a graph homomorphism $G_i\rightarrow H_i$ with $i\in [1,a]$.

\vskip 0.2cm

\textbf{The complexity of NSD-GHP:}
\begin{asparaenum}[\textrm{Comp}-1. ]
\item Since number strings are not integers, the well-known integer decomposition techniques can not be used to solve the number string decomposition problem.
\item No polynomial algorithm for cutting a number string $S=c_1c_2\cdots c_m$ with $c_i\in [0,9]$ and $S$ is not encrypted into $a_1a_2\cdots a_{3q}$, such that all $a_i$ to be correctly the elements of some matrix. As known, there are several kinds of matrices related with graphs, for example, graph adjacency matrix, Topsnut-matrix, Topcode-matrix and Hanzi-matrix, and so on.
\item If the matrix in problem has been found, it is difficult to guess the desired graphs, since it will meet NP-hard problems, such as, Graph Isomorphic Problem, and Hanzi-graph Decomposition Problem (Ref. \cite{Yao-Mu-Sun-Sun-Zhang-Wang-Su-Zhang-Yang-Zhao-Wang-Ma-Yao-Yang-Xie2019}).
\item If the desired graphs have been determined, we will face a large number of graph colorings and graph labellings for coloring exactly the desired graphs, as well as unknown problems of graph theory, such as, Graph Total Coloring Problem, Graceful Tree Conjecture.
\end{asparaenum}

\section{Conclusion}

We have defined several kinds of $W$-type totally-colored graph homomorphisms by combining graph homomorphisms and particular graph total colorings together. As known the number of particular graph total colorings is not fixed everyday, so the $W$-type totally-colored graph homomorphism is not fixed. We have constructed two kinds of $W$-type graphic lattice homomorphisms by graphic lattices based on two operations ``$\ominus$'' and ``$\odot$'', and every-zero graphic group homomorphisms from sets to sets.

Naturally, new kinds of graph homomorphisms contain new mathematical questions, such as, NSD-GHP. Our graph homomorphisms can be defined by other graph colorings/labellings not mentioned here, and our works here are just beginning of studying graph homomorphisms. It is interesting to research deeply various graph homomorphisms for network security in the ear of quantum computers.

\section*{Acknowledgment}

The author, \emph{Bing Yao}, was supported by the National Natural Science Foundation of China under grant No. 61363060 and No. 61662066. The author, \emph{Hongyu Wang}, thanks gratefully the National Natural Science Foundation of China under grants No. 61902005, and China Postdoctoral Science Foundation Grants No. 2019T120020 and No. 2018M641087.

\end{document}